\documentclass{article}[10pt]
\usepackage{amssymb}
\usepackage{amsmath, amsthm}
\usepackage{graphicx}
\newtheorem{theorem}{Theorem}[section]
\newtheorem{lemma}{Lemma}[section]
\newtheorem{corollary}{Corollary}[section]
\newtheorem{definition}{Definition}[section]
\newtheorem{proposition}{Proposition}[section]
\newtheorem{remark}{Remark}[section]

\begin{document}
\title
{A Tight Nonlinear Approximation Theory 
for Time Dependent Closed Quantum
Systems}
\author{Joseph W. Jerome\footnotemark[1]}
\date{}
\maketitle
\pagestyle{myheadings}
\markright{Approximation theory for closed quantum systems}
\medskip

\vfill
{\parindent 0cm}

{\bf 2010 AMS classification numbers:} 
35Q41; 47D08; 47H09; 47J25; 81Q05.

\bigskip
{\bf Key words:} 
Time dependent quantum systems; time-ordered evolution operators;
numerical fixed points; Faedo-Galerkin approximations. 

\fnsymbol{footnote}
\footnotetext[1]{Department of Mathematics, Northwestern University,
                Evanston, IL 60208. \\In residence:
George Washington University,
Washington, D.C. 20052}

\begin{abstract}
The approximation of fixed points by numerical fixed points was 
presented in the elegant monograph of Krasnosel'skii et al. (1972).
The theory, both in its formulation and implementation, requires a
differential operator calculus, so that its actual application has been
selective. The writer and Kerkhoven demonstrated this for the semiconductor
drift-diffusion model in 1991. In this article, we show that the theory
can be applied to   
time dependent quantum systems on bounded domains, via 
evolution operators. In addition to the kinetic operator term,  
the Hamiltonian includes 
both an external time
dependent potential and the classical nonlinear Hartree potential. 
Our result can be paraphrased a follows: For a sequence of Galerkin
subspaces, and the Hamiltonian just described, a uniquely defined sequence
of Faedo-Galerkin solutions exists; it converges in Sobolev space,
uniformly in time,  
at the maximal rate given by the projection operators. 
\end{abstract}

\section{Introduction.}
\label{Introduction}
The topic of this article is the convergence analysis for the
Faedo-Galerkin method, as applied to a model for time dependent
density functional theory (TDDFT). We consider closed quantum systems on bounded
domains, with potentials described by a time varying external potential
and by the Hartree potential. Exchange correlation potentials, for which
there is no consensus in the physics community, are not included.  
A systematic approximation theory is used
for the convergence analysis, which is spatially uniform over a specified time
interval. 

An operator calculus for approximation of 
the fixed points of nonlinear mappings
by numerical fixed points, in general Banach spaces, was developed by
Krasnosel'skii and his coworkers \cite{KrasVain}, 
and is intended
for nonlinear equations and systems. 
It was applied to 
the semiconductor model in \cite{JerKerk2} (see also \cite{Springer}) 
by the analysis of implicit fixed point mappings, obtained by Gummel
decomposition (iteration).
It was shown by the author in \cite{Farouk}
that the inf-sup theory \cite{BS94}, used for stability/convergence in finite
element applications to linear problems, is logically implied
by the Krasnosel'skii theory.
In section \ref{Kras}, particularly Theorem \ref{ther2}, 
we shall summarize the theory of \cite{KrasVain}
as it applies here. 

TDDFT was introduced into the physics community in \cite{RG}. A readable
account of the subject may be found in \cite{U}. 
It has emerged as a significant tool in application fields \cite{tddft},
especially those which study the physical properties of 
subatomic particles, since it is formulated to track electronic charge
exactly. The TDDFT model
is to be distinguished
from the nonlinear Schr\"{o}dinger equation \cite{Caz,CH}, much studied in
the mathematical community, by the presence of time dependent potentials
in the Hamiltonian. From the perspective of mathematical analysis, this
means that evolution operators, not semigroups, are the preferred tool
(for a development of these operators, due originally to Kato, see
\cite{K1,K2,J2}). In
fact, the evolution operators, together with fixed point analysis, were
used in \cite{JP,J1}, to establish strong and weak solutions, resp.
They were shown to be an effective computational tool in \cite{CP}.
In this article, we coordinate the 
explicit fixed point mappings, induced by the
evolution operators, with the operators of the Krasnosel'skii calculus.
We demonstrate the following (Th.\ \ref{mainresult}):
For a sequence of Galerkin
subspaces of $H^{1}_{0}$, with dense union, a uniquely defined sequence
of Faedo-Galerkin solutions exists for the system, 
and converges in $H^{1}_{0}$,
uniformly in time, to the unique solution $\Psi$ of the quantum system. 
The order of convergence is the same as that
given by the orthogonal projection operators, as applied
to both the solution and the initial datum. Also, $\Psi$ is the unique
fixed point of an operator $K$ defined by the evolution operator. The
principal assumption of the theory may be interpreted as a regularization   
property of the linear operator $K^{\prime}(\Psi)$.

\section{The Krasnosel'skii Calculus}
\label {Kras}

Given a fixed point
$x_{0}$ of a smooth mapping ${T}$, a numerical approximation map 
${T}_n$, with numerical fixed point $x_{n}$,
and a linear projection map ${P}_n$, a theory is constructed to estimate 
$\|x_n - {P}_n x_0\|$. 
The result 
is stated as Theorem \ref{ther2} below.
We now discuss the basic results of the theory.
 
\subsection{The abstract calculus}
Let $E$ be a Banach space and ${T}$ a mapping from an open subset $\Omega_{0}$
of $E$ into $E$.  We assume the existence of a fixed point $x_{0}$ for ${T}$:
\begin{equation}
{T}x_{0} = x_{0}.    \label{fp}
\end{equation}
If $\{E_{n}\}$ denotes a sequence of closed subspaces of $E$, 
suppose that ${T}_{n}: \Omega_{n} \mapsto E_{n},$ 
$\Omega_{n} := \Omega_{0} \cap E_{n}$, has a fixed point: 
\begin{equation}
{T}_{n}x_{n} = x_{n}.      \label{discfp}
\end{equation}
Finally, let $\{{P}_{n}\}$ be a family of 
bounded linear projections from $E$ onto $E_{n}$:
\begin{itemize}
\item 
$P_{n}^{2} = P_{n}$,
\item
$P_{n}E = E_{n}$.
\end{itemize}

We have the following \cite[Th.\ 19.1]{KrasVain}.
\begin{theorem} \label{ther2}
Let the operators ${T}$ and ${P}_{n}{T}$ be
Fr\'echet-differentiable in $\Omega_{0}$, 
and ${T}_{n}$ Fr\'echet-differentiable in $\Omega_{n}$. Assume that
{\rm (\ref{fp})} 
has a solution $x_{0} \in \Omega_{0}$ and the linear operator  
${I}  - {T}'(x_{0})$
is continuously invertible in $E$. Let
\begin{equation}
\|{P}_{n} x_{0} - x_{0}\| \rightarrow 0,
\label{projcon}
\end{equation}
\begin{equation}
\|{P}_{n}{T}{P}_{n}x_{0} - {T}x_{0}\| \rightarrow 0, 
\label{projTcon}
\end{equation}
\begin{equation}
\|{P}_{n}{T}'({P}_{n}x_{0}) - {T}'(x_{0})\| \rightarrow 0,  
\label{firstder}
\end{equation}
\begin{equation}
\|[{T}_{n} - {P}_{n}{T}]{P}_{n}x_{0}\| \rightarrow 0, 
\label{projTncon}
\end{equation}
\begin{equation}
\|[{T}'_{n} - ({P}_{n}{T})']({P}_{n}x_{0})\| \rightarrow 0, 
\label{secondder}
\end{equation}
as $n \rightarrow \infty$. Finally, assume that for any $\epsilon > 0$
there exist $n_{\epsilon}$ and $\delta_{\epsilon} > 0$ such that 
\begin{equation}
\| {T}^{'}_{n}(x) - {T}^{'}_{n}({P}_{n}x_{0}) \| \leq \epsilon
\;\;\; \mbox{for} 
\;\;\; (n \geq n_{\epsilon}; \; \|x - {P}_{n}x_{0}\| \leq
\delta_{\epsilon},  \; x \in \Omega_{n}).  \label{conderTn}
\end{equation}
Then there exist $n_{0}$ and $\delta_{0} > 0$ such that, when $n \geq n_{0}$,
equation {\rm (\ref{discfp})} has a unique solution $x_{n}$ in the ball 
$\|x - x_{0}\| \leq \delta_{0}$. Moreover,
\begin{equation}
\|x_{n} - x_{0}\| \leq \| [{I} - {P}_{n}]x_{0}\| + \|x_{n} - {
P}_{n} x_{0}\| 
\rightarrow 0 \;\;\; {\rm as} \;\;\; n \rightarrow \infty, \label{xcon}
\end{equation}
and $\|x_{n} - {P}_{n}x_{0}\|$ satisfies the following two-sided estimate
$(c_{1}, c_{2} > 0)${\rm :}
\begin{equation}
c_{1} \|{P}_{n}{T}x_{0} - {T}_{n}{P}_{n}x_{0}\| \leq
\|x_{n} - {P}_{n}x_{0}\| \leq 
c_{2} \|{P}_{n}{T}x_{0} - {T}_{n}{P}_{n}x_{0}\|.  \label{xbnd}
\end{equation}
\end{theorem}
\begin{remark}
\label{remark2.1}
In order to convert the result of the theorem into a useful form, it is
convenient to use the triangle inequality as applied to the bounds of
(\ref{xbnd}):
$$
\|P_{n}T x_{0} - T_{n} P_{n} x_{0}\| \leq
\|(P_{n} - I) Tx_{0}\| + \| Tx_{0} - TP_{n}x_{0}\| +
\|(T - T_{n})P_{n}x_{0}\|. 
$$
The rhs of this inequality dictates the convergence of the overall
approximation.
\end{remark}
\begin{remark}
The authors of \cite{KrasVain} permit the projection operators to be
unbounded; some further assumptions are required in this case. They
observe that, when $P_{n}$ is bounded, the differentiability of $P_{n}T$
follows from that of $T$, and $(P_{n}T)^{\prime}(x) = P_{n}T^{\prime}(x)$.
Our application is confined to the case when $P_{n}$ is bounded. 
\end{remark}
\section{The Model}
In its original form, 
TDDFT includes three components for the potential:  
an external potential, the Hartree potential, 
and a general non-local term representing the exchange-correlation potential, 
which is assumed to include a time history part.
The exchange-correlation term is suppressed here for the analysis, since
there is not a consensus on its representation. When included,
its required analytical
properties resemble those of the Hartree potential. 
If $\hat H$ denotes
the Hamiltonian operator of the system, then the state $\Psi(t)$ of the
system obeys the nonlinear Schr\"{o}dinger equation,
\begin{equation}
\label{eeq}
i \hbar \frac{\partial \Psi(t)}{\partial t} = \hat H \Psi(t).
\end{equation}
Here, 
$\Psi = \{\psi_{1}, \dots, \psi_{N}\}$ 
consists of 
$N$ 
orbitals, and the charge density 
$\rho$ 
is defined by 
$$ \rho({\bf x}, t) = |\Psi({\bf x}, t)|^{2} = 
\sum_{k = 1}^{N} |\psi_{k} ({\bf x}, t)|^{2}.
$$ 
An initial condition,
\begin{equation}
\label{ic}
\Psi(0) = \Psi_{0}, 
\end{equation}
and boundary conditions are included. 
The particles are confined to a bounded region 
$\Omega \subset {\mathbb R}^{3}$ 
and homogeneous Dirichlet boundary conditions hold 
within a closed system. 
$\Psi$ 
denotes a finite vector function of space and time. 
The effective potential 
$V_{\rm e}$ 
is a  real scalar function of the form,
$$
V_{\rm e} ({\bf x},t, \rho) = V({\bf x}, t) + 
W \ast \rho. 
$$
Here, 
$W({\bf x}) = 1/|{\bf x}|$ 
and the convolution
$W \ast \rho$ 
denotes the Hartree potential. If $\rho$ is extended as zero outside
$\Omega$, then, for ${\bf x} \in \Omega$, 
$$
W \ast \rho \; ({\bf x})=\int_{{\mathbb R}^{3}} 
W({\bf x} -{\bf y}) \rho({\bf y})\;d {\bf y},
$$
which depends only upon values $W({\xi})$,
$\|{\xi}\|\leq 
\mbox{diam}(\Omega)$. We may redefine $W$ 
smoothly outside this set,
so as to obtain a function of compact support for which Young's inequality
applies. 
The Hamiltonian operator is given by, 
\begin{equation}
\hat H  
= -\frac{\hbar^{2}}{2m} \nabla^{2} 
 +V({\bf x}, t) + 
W \ast \rho, 
\label{Hamiltonian1}
\end{equation}
and 
$m$ 
designates the effective mass and $\hbar$ the  
normalized Planck's constant.
\subsection{Definition of weak solution}
The solution 
$\Psi$ 
is continuous from the time interval 
$J$, 
to be
defined shortly,  
into the finite energy Sobolev space
of complex-valued 
vector functions which vanish in a generalized sense on the boundary, 
denoted 
$H^{1}_{0}(\Omega)$: $\Psi \in C(J; H^{1}_{0})$. 
The time derivative is continuous from 
$J$ 
into the dual 
$H^{-1}$ 
of 
$H^{1}_{0}$: 
$\Psi \in C^{1}(J; H^{-1})$.  
The spatially dependent test functions 
$\zeta$ 
are
arbitrary in 
$H^{1}_{0}$. 
The duality bracket is denoted 
$\langle f, \zeta \rangle$. 
Norms and inner products are discussed in the appendix.
\begin{definition}
\label{weaksolution}
For 
$J=[0,T_{0}]$,  
the vector-valued function 
$\Psi = \Psi({\bf x}, t)$ 
is a  
weak solution of (\ref{eeq}, \ref{ic}, \ref{Hamiltonian1}) if 
$\Psi \in C(J; H^{1}_{0}(\Omega)) \cap C^{1}(J;
H^{-1}(\Omega)),$ 
if 
$\Psi$ 
satisfies the initial condition 
(\ref{ic}) for 
$\Psi_{0} \in H^{1}_{0}$, 
and if 
$\forall \; 0 < t \leq T$: 
\vspace{.25in}
\begin{equation}
i \hbar\langle \frac{\partial\Psi(t)}{\partial t},
\zeta \rangle  = 
\int_{\Omega} \frac{{\hbar}^{2}}{2m}
\nabla \Psi({\bf x}, t)\cdotp \nabla { \zeta}({\bf x}) 
+ V_{\rm e}({\bf x},t,\rho) \Psi({\bf x},t) { \zeta}({\bf x})
d{\bf x}. 
\label{wsol}
\end{equation}
\end{definition}
\subsection{Hypotheses for the Hamiltonian and theorem statement}
\label{hyps}
\begin{itemize}
\item
The so-called external potential 
$V$ 
is assumed to be 
continuously
differentiable on the closure of the space-time domain. 
\end{itemize}

The following theorem was proved in \cite{J1}, based upon the evolution
operator as presented in \cite{J2}. 
\begin{theorem}
\label{EU}
For any interval 
$[0,T_{0}]$, 
the system (\ref{wsol}) in Definition
\ref{weaksolution}, 
with Hamiltonian
defined by (\ref{Hamiltonian1}),  
has a unique weak solution $\Psi$ if the hypothesis stated for $V$ holds. 
\end{theorem}
We briefly summarize the method of \cite{J1}.  
Specifically, we identify the mapping $K: C(J; H^{1}_{0}) \mapsto 
C(J; H^{1}_{0})$
for which the unique solution $\Psi$ of Theorem \ref{EU} is the
unique fixed point of the restriction of $K$ to an appropriate closed ball
on which $K$ is invariant and strictly contractive. 
\begin{definition}
\label{decomp1}
For each $\Psi^{\ast}$ in the domain
$C(J;H^{1}_{0})$ of $K$ 
we obtain the image $K \Psi^{\ast} = \Psi$ by the following
decoupling.
\begin{itemize}
\item
$\Psi^{\ast} \mapsto \rho =
 \rho({\bf x}, t) = |\Psi^{\ast}({\bf x}, t)|^{2}$.
\item
$\rho \mapsto \Psi$ by the solution of the 
{\it associated} linear problem (\ref{wsol})
where the potential $V_{\rm e}$ uses $\rho$ 
in its final argument.
\end{itemize}
\end{definition}
\begin{remark}
One can implement the second stage of the previous definition as follows.
Introduce the linear evolution
operator $U(t,s)$: for given $\Psi^{\ast}$ in
$C(J;H^{1}_{0})$, 
set $U(t,s) = U^{\rho}(t,s)$ so that
\begin{equation}
\label{evolution} 
\Psi(t) = U^{\rho}(t,0) \Psi_{0}. 
\end{equation}
For each $t$, $\Psi(t)$ 
is a function of ${\bf x}$.
Moreover, 
$\Psi = K \Psi^{\ast}$.
The properties of $K$ can be derived with the assistance of the evolution
operators. The following theorem is quoted from the results of
\cite[sections 3.3--3.5]{archive}.
\end{remark}
\begin{theorem}
\label{global}
The mapping $K$ is differentiable on $C(J; H^{1}_{0})$, with a locally
Lipschitz derivative. Moreover, at its unique fixed point $\Psi$,
the linear mapping $I - K^{\prime}(\Psi)$ is invertible with continuous
inverse. For each $\psi \in C(J; H^{1}_{0})$, 
and $\rho = |\psi|^{2}$, the derivative is given
explicitly by
\begin{equation}
K^{\prime}[\psi](\omega) = 
\frac{2i}{\hbar}\int_{0}^{t} 
U^{\rho}(t,s)\left[\mbox{\rm Re}({\bar \psi} \omega) \ast W
\right] 
U^{\rho}(s,0)\Psi_{0} \; ds.
\label{defKprime}
\end{equation}
\end{theorem}
\begin{remark}
\label{tools}
The identity,
\begin{equation}
U^{\rho_{1}}\Psi_{0}(t) - U^{\rho_{2}}\Psi_{0}(t) =  
\frac{i}{\hbar}\int_{0}^{t} U^{\rho_{1}}(t,s)[V_{\rm e}(s, \rho_{1}) -
V_{\rm e}(s, \rho_{2})]U^{\rho_{2}}(s,0)\Psi_{0} \; ds,
\label{IDENTITY2}
\end{equation}
and the inequality,
\begin{equation}
\label{LipV}
\|[V_{\rm e}({\rho_{1}})-
V_{\rm e}({\rho_{2}})]\psi\|_{C(J;H^{1}_{0})} \leq  
C\|\Psi_{1} -
\Psi_{2} 
\|_{C(J;H^{1}_{0})} \|\psi\|_{H^{1}_{0}}, 
\end{equation} 
are used to prove the preceding theorem. Here, $C$ is a locally defined
constant. 
Some adjustments of (\ref{IDENTITY2}) are also employed. Full details are
given in \cite[sect.\ 3]{archive}. The identity (\ref{IDENTITY2}) is derived by
differentiating $U^{\rho_{1}}(t,s) U^{\rho_{2}}(s,0) \Psi_{0}$ with respect
to $s$, and integrating from $0$ to $t$ after using differentiation 
properties of the
evolution operators. As derived, the operator $U^{\rho_{1}}(t,s)$, within the
integral on the rhs of the identity, is interpreted as acting on the dual
space. However, because of (\ref{LipV}),   
the operator $U^{\rho_{1}}(t,s)$ may be interpreted 
as acting invariantly on $H^{1}_{0}$ for each
fixed $s$. Finally, the identity (\ref{IDENTITY2}), 
in conjunction with (\ref{LipV}) and the
invariance properties of the evolution operator, gives the local Lipschitz
property of $K$.
\end{remark}
\section{The Projection Mappings}
The program carried out here is intended to provide the appropriate mathematical
structure to accommodate the classical Faedo-Galerkin method in
conjunction with the Krasnosel'skii calculus introduced earlier.

From Theorem \ref{global}, we see that $\Omega_{0} = E = C(J; H^{1}_{0})$
in the notation of section two. We now discuss the projections $P_{n}$.
\subsection{The approximation spaces}
\begin{definition}
\label{projectionPn}
Let $\{F_{n}\}_{n \geq 1}$ be a sequence of finite-dimensional subspaces of 
the Hilbert space $H^{1}_{0}$, 
and let $Q_{n}$
denote the orthogonal projection onto $F_{n}$ for each $n$. 
We suppose that 
$\|Q_{n} f - f \|_{H^{1}_{0}} \rightarrow 0, \; n \rightarrow \infty$,
for all $f \in H^{1}_{0}$. 
Suppose a basis $\{f_{j}, \; j = 1, \dots, k(n)\}$ is given for $F_{n}$.
If $J = [0, T_{0}]$, then 
$E_{n} \subset C(J; H^{1}_{0})$ is defined by
$$
E_{n} = \left\{ \sum_{j=1}^{k(n)} \alpha_{j}(t) f_{j}({\bf x}), \; \alpha_{j}
\in C(J), \; j = 1, \dots, k(n) \right\}.
$$
Finally, if $\psi \in C(J; H^{1}_{0})$, define $P_{n} \psi (\cdotp, t) :=
Q_{n} \psi(\cdotp, t)$, for each $t$.
%The coefficient $\alpha_{j}$ is the $j$th coefficient obtained via the
%projection $Q_{n}$.
\end{definition}
\begin{remark}
In the following proposition, we will verify the consistency of this
definition and the compatibility with the Krasnosel'skii calculus. The
subspaces $E_{n}$ are not vector subspaces, since the basis coefficients are
functions and not scalars. In the mathematical literature, $E_{n}$ is a
free module over a ring (of continuous functions). It seems to be the
appropriate concept for the analysis of the Faedo-Galerkin method.
\end{remark}
\begin{proposition}
In Definition \ref{projectionPn}:
\begin{enumerate}
\item
$\{E_{n}\}$ are closed subspaces of $C(J;H^{1}_{0})$.
\item
$P_{n}$ maps $C(J;H^{1}_{0})$ onto $E_{n}$ and $P_{n}^{2} = P_{n}$.
\end{enumerate} 
\end{proposition}
\begin{proof}
We begin with statement (1), which involves two steps.
The first is to show that $E_{n}
\subset C(J;H^{1}_{0})$.
We observe that, by use of the classical Gram-Schmidt procedure, we
may assume that the basis $\{f_{j}\}$ of $F_{n}$ is orthonormal in 
$H^{1}_{0}$. Continuity of the coefficients is preserved. 
With this property, if $t_{0} \in J$ and $t \in J$, then,
for $g \in E_{n}$:
$$
\|g(t) - g(t_{0})\|^{2}_{H^{1}_{0}} = \sum_{j=1}^{k(n)} |\alpha_{j}(t) -
\alpha_{j}(t_{0})|^{2}, 
$$ 
which has zero limit as $t \rightarrow t_{0}$. It follows that $E_{n}
\subset C(J;H^{1}_{0})$.
We now prove that $E_{n}$ is closed. 
Suppose that $\{g_{\ell}\} \subset E_{n}$ converges in the norm of $C(J;
H^{1}_{0})$ to a limit $g$. In particular,  
$\{g_{\ell}\}$ is a Cauchy sequence. 
Again, under the assumption that the basis in
$F_{n}$ is orthonormal, we write 
$$
\|g_{\ell}(t)-g_{m}(t)\|^{2}_{H^{1}_{0}}=\sum_{j = 1}^{k(n)}
|\alpha^{\ell}_{j}(t) - \alpha^{m}_{j}(t) |^{2},
$$
so that the individual coefficient sequences are 
Cauchy sequences in
the Banach space $C(J)$. It follows that the limits $\alpha_{j}(t)$ are
continuous on $J$ and that $g = \sum \alpha_{j} f_{j}$. We conclude that 
$E_{n}$ is a closed subspace of $C(J; H^{1}_{0})$. 

In order to prove statement (2), we must show that the coefficients
induced by the orthogonal projection $Q_{n}$ are continuous functions of
$t$. In fact, the difference of these
coefficients is given explicitly by the inner product, 
$$
\alpha_{j}(t) - \alpha_{j}(t_{0})  = 
(\psi(\cdotp,t)- \psi(\cdotp, t_{0}),  f_{j})_{H^{1}_{0}},
$$
which is estimated in norm by $\|\psi(\cdotp, t) - \psi(\cdot, t_{0})\|_{
H^{1}_{0}}.$ By the definition of the Banach space $C(J; H^{1}_{0})$, this
has a zero limit as $t \rightarrow t_{0}$. The remaining statements of (2) are
evident.  
\end{proof}
\subsection{The hypotheses of the Krasnosel'skii calculus: I}
\label{hypI}
The propositions of this section deal with hypotheses pertaining to
the fixed point mapping and its relation to the projections.
\begin{proposition}
\label{prop4.2}
The hypotheses (\ref{projcon}) and (\ref{projTcon}) 
%and (\ref{firstder})
hold for the case $T = K$, $x_{0} = \Psi$, 
and $P_{n}$ as defined in Definition
\ref{projectionPn}. 
Specifically,
\begin{eqnarray}
\|P_{n} \Psi - \Psi \|_{C(J; H^{1}_{0})} & \rightarrow 0 &, \label{new3}\\
\|P_{n} K P_{n} \Psi - K \Psi \|_{C(J; H^{1}_{0})} & \rightarrow 0 &,
\label{new4}
%\|P_{n} K^{\prime} (P_{n} \Psi) - K^{\prime} (\Psi) \| & \rightarrow 0 &
%\label{new5}, 
\end{eqnarray}
as $n \rightarrow \infty$.
\end{proposition}
\begin{proof}
The hypothesis (\ref{new3}) has been built into the definition of the
projection family. We rewrite (\ref{new4}) as follows, using the triangle
inequality, the fixed point property,  and the property $\|P_{n} \| = 1$.
$$
\|P_{n} K P_{n} \Psi - K \Psi \|_{C(J; H^{1}_{0})} \leq
\|P_{n} K P_{n} \Psi - P_{n} K \Psi \|_{C(J; H^{1}_{0})} + 
\|P_{n} K \Psi -  K \Psi \|_{C(J; H^{1}_{0})} \leq 
$$
$$
\|K P_{n} \Psi - K \Psi \|_{C(J; H^{1}_{0})} + 
\|P_{n}  \Psi -   \Psi \|_{C(J; H^{1}_{0})}.
$$ 
The (local Lipschitz) continuity of $K$ and (\ref{new3}) 
imply the convergence expressed
in (\ref{new4}).
%The argument for (\ref{new5}) proceeds similarly, where now the norms
%are those of the uniform operator topology. Explicitly: 
%$$
%\|P_{n} K^{\prime}(P_{n} \Psi) - K^{\prime}(\Psi) \| \leq
%\|P_{n} K^{\prime}(P_{n} \Psi) - P_{n} K^{\prime}(\Psi) \| +
%\|P_{n} K^{\prime}(\Psi) -  K^{\prime}(\Psi) \| \leq 
%$$
%$$
%%\| K^{\prime}(P_{n} \Psi) -  K^{\prime}(\Psi) \| +
%\|P_{n} K^{\prime}(\Psi) -  K^{\prime}(\Psi) \|.
%$$
%The continuity of $K^{\prime}$, the definition of $P_{n}$, and 
%the definition of the uniform operator topology now imply
%(\ref{new5}).
\end{proof}
In order to address hypothesis (\ref{firstder}), 
we require a regularization hypothesis for $K^{\prime}(\Psi)$.
\begin{definition} 
\label{reghyp}
Denote by ${\mathcal B}$ the closed unit ball of $H^{1}_{0}$, and by 
${\mathcal C}$ the set $C(J; {\mathcal B})$. Finally, set 
${\mathcal K} = K^{\prime}(\Psi) {\mathcal C}$. If the maximal dispersion
of the set ${\mathcal K}$ from $E_{n}$ is defined by
$$
{\mathcal E}({\mathcal K}; E_{n}) := \sup_{\phi \in {\mathcal K}}
\|\phi -P_{n}\phi\|_{C(J; H^{1}_{0})},
$$
then the regularization hypothesis asserts that
$$
\lim_{n \rightarrow \infty} {\mathcal E}({\mathcal K}; E_{n}) = 0.
$$
\end{definition}
Naturally, one wishes to find subspaces $F_{n} \subset H^{1}_{0}$ such
that the dispersion ${\mathcal E}$ is minimized. This represents a mild
extension of
an important topic in approximation theory, termed the $n$-width, 
which characterizes minimal dispersion from $n$-dimensional subspaces.
Here, we have the additional structure of time dependence. 
The $n$-width was introduced by A.\ N.\ Kolmogorov in \cite{Kol}, and has
been intensively studied by many authors.

According to classical approximation theory, one expects the
regularization hypothesis to hold if members of ${\mathcal K}$
have spatial regularity exceeding that for $H^{1}_{0}$. The exact formula
for $K^{\prime}(\Psi)$, given by (\ref{defKprime}), allows some observations. 
For example, the middle (convolution) term is in $C(J; H^{2})$, as can be
shown by distributing the second order derivative, and then applying the
general Young's inequality. There is a bound dependent only on that 
of $\omega$.
Moreover, it was shown in \cite{JP} that 
$U^{\rho}(s,0) \Psi_{0}$ is in $C(J; H^{2} \cap H^{1}_{0})$ if $\Psi_{0}
\in H^{2} \cap H^{1}_{0}$. One expects a product of such $H^{2}$
functions to be in a fractional order Sobolev space with index greater
than one. This means that the hypothesis is not unrealistic. 
\begin{proposition}
Suppose the regularization hypothesis of Definition \ref{reghyp} holds. Then 
(\ref{firstder}) holds for the case
$T = K$, $x_{0} = \Psi$, 
and $P_{n}$ as defined in Definition
\ref{projectionPn}. We have: 
\begin{equation}
\label{new5} 
\|P_{n} K^{\prime} (P_{n} \Psi) - K^{\prime} (\Psi) \|  \rightarrow 0,
\; n \rightarrow \infty. 
\end{equation}
\end{proposition}
\begin{proof}
The norms used here
are those of the uniform operator topology. Explicitly: 
$$
\|P_{n} K^{\prime}(P_{n} \Psi) - K^{\prime}(\Psi) \| \leq
\|P_{n} K^{\prime}(P_{n} \Psi) - P_{n} K^{\prime}(\Psi) \| +
\|P_{n} K^{\prime}(\Psi) -  K^{\prime}(\Psi) \|.
$$
The continuity of $K^{\prime}$, the definition of $P_{n}$, and 
the definition of the uniform operator topology imply the estimate for the
first term. The second term requires the regularization hypothesis of
Definition \ref{reghyp}.
Thus, we obtain (\ref{new5}).
\end{proof}
\section{The Numerical Fixed Point Map}
We identify the mapping $K_{n}: E_{n} \mapsto 
E_{n}$ which serves as the numerical fixed point map. It is the analog of
the mapping $K$ of Definition \ref{decomp1}. 
\begin{definition}
\label{decomp2}
For each $u^{\ast}$ in the domain
$E_{n}$ of $K_{n}$ 
we obtain the image $K_{n} u^{\ast} = u \in E_{n}$ by the following
decoupling.
\begin{itemize}
\item
$u^{\ast} \mapsto \rho =
 \rho({\bf x}, t) = |u^{\ast}({\bf x}, t)|^{2}$.
\item
$\rho \mapsto u$ by the solution of the linear 
system (\ref{step2linFGal}) to follow, 
where the potential $V_{\rm e}$ uses $\rho$ 
in its final argument.
\end{itemize}
\end{definition}
\subsection{The Galerkin operator}
\label{GO}
\begin{definition}
\label{absCgal}
Let $F_{n}  \subset H^{1}_{0}(\Omega)$ be 
a given finite dimensional linear subspace, with positive dimension $k(n)$. 
Define, for each fixed $t \in [0,T_{0}]$, and 
$u(\cdotp, t) \in F_{n}$, 
the relation,
\begin{equation}
\label{Galerkin}
{\mathcal G}(t, u)[v] = 
\int_{\Omega} \left\{\frac{{\hbar}^{2}}{2m}
\nabla u({\bf x},t)\cdotp \nabla {v}({\bf x}) 
+ V_{\rm e}({\bf x},t, \rho) u({\bf x},t) {v}({\bf x})
\right\}d{\bf x}, \; 
\forall v \in F_{n}, 
\end{equation}
where $\rho = |u^{\ast}|^{2}$, for a given $u^{\ast} \in E_{n}$.
Finally, $u \in E_{n}$ is a solution of the linear 
Faedo-Galerkin equation if 
\begin{equation}
\label{step2linFGal}
i \hbar \frac{\partial u}{\partial t} 
= {\mathcal G}(t, u), \;
u(\cdotp, 0) = Q_{n} \Psi_{0},
\end{equation}
for $\partial u/\partial t$ a continuous linear functional on $E_{n}$
Here, $Q_{n}$ retains its meaning as the orthogonal projection onto
$F_{n}$. In the event that $\rho = |u|^{2}$, we say that $u$ is a solution
of the nonlinear Faedo-Galerkin equation.
\end{definition}
\begin{remark}
\label{remark5.1}
There are three separate points of analysis remaining for the numerical
fixed point map $K_{n}$. We state them here.
\begin{itemize}
\item
Definition \ref{decomp2} is consistent. In particular, this entails a
liner theory for equation (\ref{step2linFGal}). 
\item
$K_{n}$ is continuously differentiable.
\item
Convergence properties (\ref{projTncon}), (\ref{secondder}), 
and (\ref{conderTn}) hold in the present context.
\end{itemize}
\end{remark}
We take these up in the following subsections. Note that the existence of
numerical fixed points is a consequence of the theory; it is {\it not}
necessary to establish this independently.
\subsection{The linear problem: Faedo-Galerkin evolution operators} 
The Galerkin operator ${\mathcal G}$ and the Galerkin Cauchy problem
(\ref{step2linFGal}) are solved by the evolution operators associated with
the subspace $E_{n}$, where $E_{n}$ has the meaning of
Definition \ref{projectionPn}. 
A less general potential, consisting only of an external
potential, was considered in \cite{JF} with simpler estimates.
\begin{theorem}
\label{lacp}
The linear approximate Cauchy problem (\ref{step2linFGal}) 
is solvable for $u = \Psi_{\mathcal G}$  
by the formula,
$$
\Psi_{\mathcal G}(\cdotp,t) = U_{{\mathcal G}}(t, 0) Q_{n}\Psi_{0}.
$$
Here, $U_{{\mathcal G}}(t,s)$ denotes the time-ordered evolution operator, 
which acts
invariantly on $F_{n}$ and is strongly
differentiable in both arguments. 
As a function of ${\bf x}$ and $t$,
$\Psi_{\mathcal G}$ is in $E_{n}$ and is interpreted as  
the linear 
Faedo-Galerkin
approximation.
\end{theorem}
\begin{proof}
The generation of the evolution operators from a family of stable,
strongly continuous semigroups on a frame space $X$ is discussed in 
\cite[Ch.\ 6]{J2}, based on Kato's original results in \cite{K1,K2}. 
For the model
considered here, and, indeed, a somewhat more general one, the details are
presented in \cite{J1} for the family $(-i/\hbar) \{\hat H(t) \}$ of
generators of contractive semigroups $\{W(t,s)\}$  
on $H^{-1}$. These semigroups are
shown to act stably on $H^{1}_{0}$ via a similarity relation which is a
cornerstone of Kato's theory. Equivalently, this is expressed as a
commutator relation. 
The proof here amounts to the verification 
that the generators $(-i/\hbar) \{{\mathcal G}(t) \}$ enjoy similar
properties on $F_{n}$ and its dual, resp. We may obtain these semigroups,
denoted $\{W_{{\mathcal G}}(t,s)\}$, from the family $\{W(t,s)\}$ as follows.   
If $\ell$ is a (continuous) linear functional on $F_{n}$, extend $\ell$ to 
${\hat \ell} \in H^{-1}$ by defining ${\hat \ell}$ to be zero on the
orthogonal complement of $F_{n}$ in $H^{1}_{0}$. We can define
$$
W_{{\mathcal G}}(t,s)[\ell] := W(t,s)[{\hat \ell} \;]\;|_{F_{n}}.
$$ 
The remainder of the proof deals with the invariance of 
$\{W_{{\mathcal G}}(t,s)\}$ on $F_{n}$.
We recall that for the full Hamiltonian, this means that 
\begin{equation}
\label{simH}
S {\hat H}(t) S^{-1} = {\hat H}(t) + B(t),
\end{equation}
where $S$ is an isomorphism from $H^{1}_{0}$ to $H^{-1}$, and
where $B(t)$ represents a bounded operator on (all of) $H^{-1}$,
with norm uniformly estimated in $t$.
It was shown in \cite{J1} that, if $S$ is the canonical isomorphism from 
$H^{1}_{0}$ to $H^{-1}$, then 
$B(t)$ is given explicitly  
on $H^{-1}$ by:
\begin{equation}
\langle B(t) \ell, \zeta \rangle 
 = \int_{\Omega} \left[\frac{\hbar^{2}}{2m}\psi\nabla V_{ \rm e}  
\cdotp \nabla { \zeta} \right]  
\; d{\bf x}, \; \forall \zeta \in H^{1}_{0},
\label{BH}
\end{equation}
where $\psi = S^{-1} \ell$.
For the Faedo-Galerkin case, one replaces (\ref{simH}) by
\begin{equation}
\label{simG}
S_{{\mathcal G}} {\mathcal G}(t) S^{-1}_{{\mathcal G}}  = 
{\mathcal G}(t) + B_{{\mathcal G}}(t),
\end{equation}
where $S_{{\mathcal G}}$ is an induced isomorphism from 
$F_{n}$ to its dual, and (\ref{BH}) is replaced by
\begin{equation}
\langle B_{{\mathcal G}}(t) \ell, \zeta \rangle 
 = \int_{\Omega} \left[\frac{\hbar^{2}}{2m}\psi\nabla V_{ \rm e}  
\cdotp \nabla { \zeta} \right]  
\; d{\bf x}, \; \forall \zeta \in F_{n},
\label{BG}
\end{equation}
where $\psi = S^{-1}_{{\mathcal G}}\ell$.
The estimate is obtained by the generalized H\"{o}lder inequality,
with $\psi \in L^{6}, \nabla V_{\rm e} \in L^{3}, \nabla \zeta \in L^{2}$.
In particular, we obtain the evolution operator
$U_{{\mathcal G}}(t,s)$ and the corresponding theory for the Cauchy
problem. 
\end{proof}
\begin{remark}
\label{extensionofUGal}
The $L^{3}$ property of $\nabla V_{\rm e}$ is established in
\cite{archive} in Lemma 3.1 and its proof.
It is possible to extend $U_{{\mathcal G}}$ to all of $C(J; H^{1}_{0})$ as
follows: For fixed $t$, for $s < t$, and $\psi \in C(J; H^{1}_{0})$, set 
$$
U_{{\mathcal G}}(t,s) \psi (\cdotp, s) :=  
U_{{\mathcal G}}(t,s) Q_{n} \psi (\cdotp, s).
$$
In words, we consider the action only on the orthogonal projection of a
given function at time $s$.

For later use,
we observe that the same reasoning can be applied to a construction
of the evolution operators associated with $I - Q_{n}$. 
\end{remark}
\subsection{Differentiability of $K_{n}$}
There are two steps in the analysis of the differentiability of $K_{n}$: 
\begin{itemize}
\item
The verification of the existence of the G\^{a}teaux derivative at each member
of $C(J; E_{n})$ with its accompanying formula; 
\item
The verification of the continuity of the G\^{a}teaux derivative,
hence the existence of the continuous  
Fr\'{e}chet derivative with the same formula.
\end{itemize}
We 
begin with a preliminary lemma. 
\begin{lemma}
\label{uniform}
Suppose that $\psi_{\epsilon}$ converges to $\psi$ in $E_{n}$
as $\epsilon \rightarrow 0$. Then 
$U_{{\mathcal G}}^{\rho_{\epsilon}}$ converges to
$U_{{\mathcal G}}^{\rho}$ in the operator topology, uniformly in $t,s$.
In fact, the convergence is of order 
$O(\|\psi_{\epsilon} - \psi\|_{C(J;H^{1}_{0})})$.
\end{lemma}
\begin{proof}
The proof adapts the proof of \cite[Lemma
3.2]{archive}. The version of (\ref{IDENTITY2}) appropriate here is
the operator equation,
\begin{equation}
U_{{\mathcal G}}^{\rho_{1}}(t,s) - U_{{\mathcal G}}^{\rho_{2}}(t,s) =  
\frac{i}{\hbar}\int_{s}^{t} U_{{\mathcal G}}^{\rho_{1}}(t,r)
[V_{\rm e}(r, \rho_{1}) -
V_{\rm e}(r, \rho_{2})]U_{{\mathcal G}}^{\rho_{2}}(r,s)\; dr.
\label{identity2}
\end{equation}
If we identify $\rho_{1}$ with $\rho_{\epsilon}$ in the formula, and 
$\rho_{2}$ with $\rho$, then the convergence follows from (\ref{LipV}); 
note that 
$U_{{\mathcal G}}^{\rho_{\epsilon}}(t,s)$ is uniformly bounded in
$\epsilon$, as explained in the following remark. 
\end{proof}
\begin{remark}
The uniform boundedness of
$U_{{\mathcal G}}^{\rho_{\epsilon}}(t,s)$ can be derived from the
construction of the evolution operators (see \cite{J2}). Specifically, a
bound can be obtained  
by appropriately exponentiating 
a bound for (\ref{BG}) over the global time interval.
\end{remark}
\begin{proposition}
The numerical fixed point mapping $K_{n}$ is G\^{a}teaux differentiable
on $E_{n}$ with derivative,
\begin{equation}
K^{\prime}_{n}[\psi](\omega) = 
\frac{2i}{\hbar}\int_{0}^{t} 
U^{\rho}_{\mathcal G}(t,s)\left[\mbox{\rm Re}({\bar \psi} \omega) \ast W
\right] 
U^{\rho}_{\mathcal G}(s,0) Q_{n}\Psi_{0} \; ds.
\label{defKnprime}
\end{equation}
Here, $\psi$ and $\omega$ are arbitrary in $E_{n}$ and $\rho =
|\psi|^{2}$.
In fact, the derivative is continuous, hence $K_{n}$ is
continuously Fr\'{e}chet differentiable.
\end{proposition}
\begin{proof}
Let $\psi$ be a given element of $E_{n}$ and set $\rho = |\psi|^{2}$.
Set
$$\rho_{\epsilon} =  
|\psi + \epsilon \omega|^{2}, \;  \mbox{for} \; \omega \in E_{n}, \;
\epsilon \in {\mathbb R}, \epsilon \not=0. $$
We will use (\ref{identity2}) as applied to $Q_{n} \Psi_{0}$:
\begin{eqnarray}
&U_{{\mathcal G}}^{\rho_{\epsilon}}Q_{n} \Psi_{0}(t) - 
U_{{\mathcal G}}^{\rho}Q_{n} \Psi_{0}(t) = \nonumber \\ 
&\frac{i}{\hbar}\int_{0}^{t} U_{{\mathcal G}}^{\rho_{\epsilon}}(t,s)
[V_{\rm e}(s, \rho_{\epsilon}) -
V_{\rm e}(s, \rho)]U_{{\mathcal G}}^{\rho}(s,0)Q_{n} \Psi_{0} \; ds.
\label{IDENTITY3}
\end{eqnarray}
This identity 
and the inequality (\ref{LipV}) are used as in \cite[sections 3.3-3.4]{archive}.
By direct calculation, we obtain
$$
\frac{U_{{\mathcal G}}^{\rho_{\epsilon}}Q_{n}\Psi_{0}(t) - 
U_{{\mathcal G}}^{\rho}Q_{n}\Psi_{0}(t)}{\epsilon} = 
\frac{2i}{\hbar}\int_{0}^{t} 
U_{{\mathcal G}}^{\rho_{\epsilon}}(t,s)
\left[\mbox{Re}({\bar \psi} \omega) \ast W \right]
U_{{\mathcal G}}^{\rho}(s,0)Q_{n}\Psi_{0} \; ds \; +
$$
$$
\frac{i\epsilon}{\hbar}\int_{0}^{t} 
U_{{\mathcal G}}^{\rho_{\epsilon}}(t,s)\left[|\omega|^{2} \ast W\right] 
U_{{\mathcal G}}^{\rho}(s,0) Q_{n}\Psi_{0} \; ds. 
$$
In \cite{archive}, this limit was computed in the $C(J; H^{1}_{0})$ topology.
Here, we may use the same topology. 
Specifically, the first term
converges to the derivative (cf.\ Lemma \ref{uniform}), 
and the second term converges to zero. 
Note that the multiplier of $\epsilon$ remains bounded.

The continuity of $K_{n}^{\prime}$ can be inferred directly from the
formula (\ref{defKnprime}) by using Lemma \ref{uniform} 
together 
with Lemma 3.1 of \cite{archive}.
\end{proof}
%\begin{remark}
%\label{KnprimeLip}
%In fact, by using the proof of \cite[Proposition 3.2]{archive}, the
%continuity of 
%$K_{n}^{\prime}$ can be extended to uniform (with respect to $n$) local
%Lipschitz continuity. We will not use this property in the article, since
%proof details have not been presented. 
%\end{remark}
\subsection{The hypotheses of the Krasnosel'skii calculus: II}
\label{hypII}
We will prove 
convergence properties (\ref{projTncon}), (\ref{secondder}), 
and (\ref{conderTn}) in turn.
However, we begin with a fundamental proposition on the approximation of
$K$ by $K_{n}$. 
The following proposition is required for the proofs of 
Lemma \ref{Tncon} and 
Theorem
\ref{mainresult} to follow. 
\begin{proposition}
\label{KnapproximatesK}
If $\Psi$ is the unique fixed point of $K$, given by
$\Psi = U^{\rho}(t,0) \Psi_{0}$, then
$$
\|K_{n} P_{n} \Psi - K P_{n} \Psi \|_{C(J; H^{1}_{0})} \rightarrow 0, \;
\mbox{as} \;n \rightarrow \infty.
$$
In fact, the convergence is of the order,
$$
O(\|P_{n} \Psi - \Psi \|_{C(J; H^{1}_{0})}) + 
 O(\|Q_{n} \Psi_{0} - \Psi_{0} \|_{H^{1}_{0}}), 
$$
as $n \rightarrow \infty$.
\end{proposition}
\begin{proof}
Set $\rho_{n} = |P_{n}\Psi|^{2}$.
By the definitions of $K$ and $K_{n}$, we have
$$
K P_{n} \Psi 
- K_{n} P_{n} \Psi 
= U^{\rho_{n}}\Psi_{0} - U^{\rho_{n}}_{{\mathcal G}}Q_{n} \Psi_{0}. 
$$
We write this difference as the sum,
$$
K P_{n} \Psi 
- K_{n} P_{n} \Psi 
= U^{\rho_{n}}[\Psi_{0} - Q_{n} \Psi_{0}] + 
[U^{\rho_{n}} - U^{\rho_{n}}_{{\mathcal G}}] Q_{n} \Psi_{0}. 
$$
By the boundedness of the evolution operators, the first term converges
to zero with order,
$$
 O(\|\Psi_{0} - Q_{n} \Psi_{0} \|_{H^{1}_{0}}). 
$$
We analyze the second difference, by writing it as the sum,
\begin{equation}
\label{twoterms}
[U^{\rho_{n}} - U^{\rho_{n}}_{{\mathcal G}}] Q_{n} \Psi_{0} 
= 
[P_{n} U^{\rho_{n}} - U^{\rho_{n}}_{{\mathcal G}}] Q_{n} \Psi_{0} 
+\;
(I -P_{n}) U^{\rho_{n}} Q_{n} \Psi_{0}. 
\end{equation}
The first rhs term is zero, as we now show. We will require the
following representation to show this:
\begin{eqnarray}
&P_{n}U^{\rho_{n}}(t,0)Q_{n} \Psi_{0} -U^{\rho_{n}}_{{\mathcal
G}}(t,0)Q_{n}\Psi_{0}  =  
P_{n}[U^{\rho_{n}}(t,0)Q_{n} \Psi_{0} -U^{\rho_{n}}_{{\mathcal
G}}(t,0)Q_{n}\Psi_{0}]\nonumber \\
%\begin{equation}
%\label{rep1}
&=\left(\frac{i}{\hbar} \right) P_{n} \int_{0}^{t} U^{\rho_{n}}(t,r)
[{\hat H}(\rho_{n}) - {\mathcal G}(\rho_{n})]
U^{\rho_{n}}_{{\mathcal G}}(r,0)Q_{n} \Psi_{0} \; dr. 
\label{rep1}
\end{eqnarray}
Now the difference,  
${\hat H}(\rho_{n}) - {\mathcal G}(\rho_{n})$,
for each fixed $t$,
is a continuous linear functional which vanishes on $F_{n}$.
%Since this operator,
%hence the integrand of (\ref{rep1}), vanishes as a
%continuous linear functional on $E_{n}$, we can draw the desired conclusion.
We may therefore regard 
$U^{\rho_{n}}(t,s)$ as acting invariantly on the orthogonal complement of 
$F_{n}$ in $H^{1}_{0}$. Time integration preserves the invariance, and the
result is annihilated by $P_{n}$. 
We write the second term on the rhs of (\ref{twoterms}) as,
\begin{eqnarray}
(I -P_{n}) U^{\rho_{n}}(t,0) Q_{n} \Psi_{0} &=& 
(I -P_{n})[ U^{\rho_{n}}(t,0)- U^{\rho}(t,0)] Q_{n} \Psi_{0} \nonumber \\ 
&+& (I -P_{n}) U^{\rho}(t,0) Q_{n} \Psi_{0}.
\label{33}
\end{eqnarray}
According to \cite[Lemma 3.2]{archive}, the convergence of  
$\| U^{\rho_{n}}(t,0)- U^{\rho}(t,0)\|$ in the operator norm is of order,
$O(\|\Psi - P_{n} \Psi\|_{C(J;H^{1}_{0})})$.
The remaining term of (\ref{33}) is written as,
\begin{equation}
\label{last}
(I -P_{n}) U^{\rho}(t,0) Q_{n} \Psi_{0} =
(I -P_{n}) U^{\rho}(t,0) \Psi_{0} +
(I -P_{n}) U^{\rho}(t,0) (Q_{n} - I)\Psi_{0}. 
\end{equation}
The first term on the rhs of (\ref{last}) 
is rewritten as $(I - P_{n}) \Psi$; its
estimation is obvious. The second term is of order, 
$
 O(\|Q_{n} \Psi_{0} - \Psi_{0} \|_{H^{1}_{0}}). 
$
This completes the proof.
\end{proof}
\begin{lemma}
\label{Tncon}
The limit relation (\ref{projTncon}) holds in the case
$$
T_{n} \mapsto K_{n}, T \mapsto K, x_{0} \mapsto \Psi.
$$
More precisely,
\begin{equation}
\label{new6}
\|P_{n} K P_{n} \Psi - K_{n} P_{n} \Psi \|_{C(J; H^{1}_{0})}
\rightarrow 0, \; n \rightarrow \infty. 
\end{equation}
Here, $\Psi$ designates the unique fixed point of $K$ in $C(J;
H^{1}_{0})$, and the projection $P_{n}$ is defined in Definition
\ref{projectionPn}.
\end{lemma}
\begin{proof}
We begin with the triangle inequality, and estimate the resulting terms.
$$
\|P_{n} K P_{n} \Psi - K_{n} P_{n} \Psi \|_{C(J; H^{1}_{0})} \leq
\|P_{n} K P_{n} \Psi - K \Psi \|_{C(J; H^{1}_{0})} + 
$$
$$
\|K \Psi - K P_{n} \Psi \|_{C(J; H^{1}_{0})} \;+ 
\|K P_{n} \Psi - K_{n} P_{n} \Psi \|_{C(J; H^{1}_{0})}. 
$$
The first term has already been estimated. It is relation (\ref{new4}), 
as proved in Proposition \ref{prop4.2}.
The third term is estimated by Proposition
\ref{KnapproximatesK}. The second term is estimated by the local Lipschitz
continuity of $K$. This concludes the proof. 
\end{proof}
\begin{lemma}
\label{LF}
The limit relation 
(\ref{secondder}) 
holds, with the same identifications
as the previous lemma. Specifically,
\begin{equation}
\label{new7}
\|[K_{n}^{\prime} - P_{n}K^{\prime}](P_{n} \Psi)\| \rightarrow 0, \;
\mbox{as} \; n \rightarrow \infty.
\end{equation}
\end{lemma}
\begin{proof}
To clarify the notation, set $\psi_{n} = P_{n} \Psi, \rho_{n} =
|\psi_{n}|^{2}$. For
arbitrary $\omega \in E_{n}, \|\omega\|_{C(J;H^{1}_{0})} \leq 1$, 
we write the difference of
$P_{n}K^{\prime}[\psi_{n}](\omega)$ and 
$K^{\prime}_{n}[\psi_{n}](\omega)$ as: 
$$
P_{n} K^{\prime}[\psi_{n}](\omega) -
K^{\prime}_{n}[\psi_{n}](\omega) =
$$
$$
  \frac{2i}{\hbar}\int_{0}^{t} 
P_{n} U^{\rho_{n}}(t,s)\left[\mbox{Re}({\bar \psi_{n}} \omega) \ast W \right]
 U^{\rho_{n}}(s,0) (I - Q_{n})\Psi_{0} \; ds \;+
$$
$$
  \frac{2i}{\hbar}\int_{0}^{t} 
 P_{n} U^{\rho_{n}}(t,s)\left[\mbox{Re}({\bar \psi_{n}} \omega) \ast W \right]
 U^{\rho_{n}}(s,0) Q_{n}\Psi_{0} \; ds \;-
$$
$$
  \frac{2i}{\hbar}\int_{0}^{t} 
U^{\rho_{n}}_{{\mathcal G}}(t,s)\left[\mbox{Re}({\bar \psi_{n}} \omega) \ast W
\right]
U^{\rho_{n}}_{{\mathcal G}}(s,0) Q_{n} \Psi_{0} \; ds.
$$
Note that we commuted $Q_{n}$ and $\int_{0}^{t}$, and then replaced
$Q_{n}$ by $P_{n}$.
Because of the 
boundedness of the evolution operator on $H^{1}_{0}$, the
first term is estimated by a constant times $\|\Psi_{0} - Q_{n} \Psi_{0}
\|_{H^{1}_{0}}$. Here, we directly use \cite[Lemma 3.1]{archive}.
The indicated difference between the second and third terms  
can be written as the sum of the two differences,
$d_{1}, d_{2}$, where $d_{1} =$
$$
  \frac{2i}{\hbar}\int_{0}^{t} 
P_{n}U^{\rho_{n}}(t,s)\left[\mbox{Re}({\bar \psi_{n}} \omega) \ast W \right]
( U^{\rho_{n}}(s,0)-U^{\rho_{n}}_{{\mathcal G}}(s,0))Q_{n} \Psi_{0} \; ds, 
$$
and where $d_{2} =$
$$
  \frac{2i}{\hbar}\int_{0}^{t} 
[P_{n}U^{\rho_{n}}(t,s)-
U^{\rho_{n}}_{{\mathcal G}}(t,s)]
\left[\mbox{Re}({\bar \psi_{n}} \omega) \ast W \right]
U^{\rho_{n}}_{{\mathcal G}}(s,0)Q_{n}\Psi_{0} \; ds.
$$
Now $d_{1}$ can be estimated by using the proof of Proposition 
\ref{KnapproximatesK}, beginning with (\ref{twoterms}), where it was shown that 
$$
\|(U^{\rho_{n}}-U^{\rho_{n}}_{{\mathcal G}})Q_{n} \Psi_{0}
\|_{C(J;H^{1}_{0})} \rightarrow 0, \; n \rightarrow \infty, 
$$
with order of convergence, 
$O(\|\Psi - P_{n} \Psi\|_{C(J;H^{1}_{0})}) + O(\|\Psi_{0} - Q_{n} \Psi_{0}
\|_{H^{1}_{0}})$.
This is maintained for $d_{1}$ via \cite[Lemma 3.1]{archive}.
Also, $d_{2} = 0$. We show this by writing the term as the sum, 
$$
  \frac{2i}{\hbar}\int_{0}^{t} 
[P_{n}U^{\rho_{n}}(t,s)-
U^{\rho_{n}}_{{\mathcal G}}(t,s)]
\left[\mbox{Re}({\bar \psi_{n}} \omega) \ast W \right]
U^{\rho_{n}}_{{\mathcal G}}(s,0)Q_{n}\Psi_{0} \; ds =
$$
$$
  \frac{2i}{\hbar}\int_{0}^{t} 
[P_{n}U^{\rho_{n}}(t,s)-
U^{\rho_{n}}_{{\mathcal G}}(t,s)]
P_{n}\left[\mbox{Re}({\bar \psi_{n}} \omega) \ast W \right]
U^{\rho_{n}}_{{\mathcal G}}(s,0)Q_{n}\Psi_{0} \; ds +
$$
$$
  \frac{2i}{\hbar}\int_{0}^{t} 
[P_{n}U^{\rho_{n}}(t,s)-
U^{\rho_{n}}_{{\mathcal G}}(t,s)]
(I - P_{n})\left[\mbox{Re}({\bar \psi_{n}} \omega) \ast W \right]
U^{\rho_{n}}_{{\mathcal G}}(s,0)Q_{n}\Psi_{0} \; ds. 
$$
Since the operator,
$$
P_{n}U^{\rho_{n}}(t,s)-
U^{\rho_{n}}_{{\mathcal G}}(t,s),
$$
vanishes on $E_{n}$ and on the subspace, which, for each fixed time, is
the orthogonal complement of $F_{n}$, we draw the conclusion that $d_{2} =
0$. 
%The dual norm may be used to estimate the
%entire expression. 

This concludes the proof.
\end{proof}
The proof of this result allows a generalization to the case where $\psi_{n}$ 
is in a closed $E_{n}$ neighborhood of $P_{n} \Psi$. This will be required
for the verification of the hypothesis (\ref{new8}) to follow.
\begin{corollary}
\label{LFG}
Let $\epsilon > 0$ be specified. Then there is a real number
$\delta_{\epsilon} > 0$ and an integer $n_{\epsilon}$ such that, for
$$
\|\psi_{n} - P_{n} \Psi \|_{C(J; H^{1}_{0})} \leq \delta_{\epsilon}, \;
\psi_{n} \in E_{n}, \; n \geq n_{\epsilon},
$$
then
\begin{equation}
\label{alt7}
\|[K_{n}^{\prime} - P_{n}K^{\prime}](\psi_{n})\| \leq \epsilon. 
\end{equation}
\end{corollary}
\begin{proof}
We telescope the representations of the previous proof as follows. Note
that $\psi_{n}$ and $\rho_{n}$ have the new interpretation as specified in
the corollary. 
$$
P_{n} K^{\prime}[\psi_{n}](\omega) -
K^{\prime}_{n}[\psi_{n}](\omega) =
$$
$$
  \frac{2i}{\hbar}\int_{0}^{t} 
P_{n} U^{\rho_{n}}(t,s)\left[\mbox{Re}({\bar \psi_{n}} \omega) \ast W \right]
 U^{\rho_{n}}(s,0) (I - Q_{n})\Psi_{0} \; ds \;+
$$
$$
  \frac{2i}{\hbar}\int_{0}^{t} 
P_{n}U^{\rho_{n}}(t,s)\left[\mbox{Re}({\bar \psi_{n}} \omega) \ast W \right]
( U^{\rho_{n}}(s,0)-U^{\rho_{n}}_{{\mathcal G}}(s,0))Q_{n} \Psi_{0} \; ds
\; + 
$$
$$
  \frac{2i}{\hbar}\int_{0}^{t} 
[P_{n}U^{\rho_{n}}(t,s)-
U^{\rho_{n}}_{{\mathcal G}}(t,s)]
\left[\mbox{Re}({\bar \psi_{n}} \omega) \ast W \right]
U^{\rho_{n}}_{{\mathcal G}}(s,0)Q_{n}\Psi_{0} \; ds.
$$
As before, the first term is estimated by a constant times
$\|\Psi_{0} - Q_{n} \Psi_{0}\|_{H^{1}_{0}}$, and the third term is zero. 
The only difference is 
the second term,  which reduces to a study of the term, 
$$
\|(U^{\rho_{n}}-U^{\rho_{n}}_{{\mathcal G}})Q_{n} \Psi_{0}
\|_{C(J;H^{1}_{0})}, 
$$
with the new interpretation of $\rho_{n}$.
This requires a new analysis of the second term on the rhs of 
(\ref{twoterms}) which we rewrite as
follows. 
$$
(I -P_{n}) U^{\rho_{n}} Q_{n} \Psi_{0} = 
[(I -P_{n}) U^{\rho_{n}} Q_{n} \Psi_{0} - 
(I -P_{n}) U^{|P_{n} \Psi|^{2}} Q_{n} \Psi_{0}] + 
$$
$$
(I -P_{n}) U^{|P_{n} \Psi|^{2}} Q_{n} \Psi_{0}. 
$$
When the triangle inequality is applied, (\ref{last}) leads to 
an estimate, previously derived, for the final term.
The difference expression is a term of order 
$O(\|\psi_{n} - {P}_{n}\Psi\|_{C(J; H^{1}_{0}})$, which will satisfy a bound of
$\epsilon/2$ by  
choice of $\delta_{\epsilon}$. Now
$n_{\epsilon}$ can be chosen according to the remaining terms, so that an
additional $\epsilon/2$ is added to the estimate. 
This concludes the proof.
\end{proof}
\begin{lemma}
\label{LF2}
The limit relation 
(\ref{conderTn}) 
holds, with the same identifications
as the previous lemmas.
Specifically, for any $\epsilon > 0$,
there exist $n_{\epsilon}$ and $\delta_{\epsilon} > 0$ such that 
\begin{equation}
\label{new8}
\| {K}^{'}_{n}(\psi) - {K}^{'}_{n}({P}_{n}\Psi) \| \leq \epsilon
\;\;\; \mbox{for} 
\;\;\; (n \geq n_{\epsilon}; \; \|\psi - {P}_{n}\Psi\| \leq
\delta_{\epsilon},  \; \psi \in E_{n}).  
\end{equation}
\end{lemma}
\begin{proof}
We write, for $\psi$ in a neighborhood of $P_{n} \Psi$ in $E_{n}$ as yet
to be determined, 
$$
K_{n}^{\prime}(\psi) - K_{n}^{\prime}(P_{n} \Psi)
 = [K_{n}^{\prime}(\psi) - P_{n} K^{\prime}(\psi)] \;+
$$
$$
[P_{n} K^{\prime}(\psi) - P_{n} K^{\prime}(P_{n} \Psi)] \;+
[P_{n} K^{\prime}(P_{n} \Psi) - K^{\prime}_{n}(P_{n} \Psi)]. 
$$
To estimate the second  term, choose a Lipschitz constant $C$ such that
$K^{\prime}$ is Lipschitz continuous, with this constant, in the closed
ball of radius one centered at $\Psi$ in $C(J; H^{1}_{0})$. 
Choose $N$ sufficiently large so that, for $n \geq N$,
$P_{n} \Psi$ is in the concentric ball of radius $1/2$. 
Then select
$\delta_{\epsilon}^{\prime}$ as follows:
$$
\delta_{\epsilon}^{\prime} = \min \left(1/2, \; \frac{\epsilon}{3C} \right).
$$ 
We can choose $n_{\epsilon}$ and $\delta_{\epsilon}^{\prime \prime}$
so that $n_{\epsilon} \geq N$ and the
first and third terms do not exceed $\epsilon/3$ for 
$n \geq n_{\epsilon}$. 
This is possible by an application of Lemma \ref{LF} to the third
expression, and Corollary \ref{LFG} to the first expression. 
We conclude that if
$$
\delta_{\epsilon} := \min(\delta_{\epsilon}^{\prime}, 
\delta_{\epsilon}^{\prime \prime}),
$$
then
condition (\ref{new8}) is satisfied, and the proof is concluded.
\end{proof}
\subsection{The main result}
\begin{theorem}
\label{mainresult}
Suppose that $\Psi$ is the unique solution satisfying Definition
\ref{weaksolution} and guaranteed by Theorem \ref{EU}.
Suppose that the subspaces $E_{n} \subset C(J; H^{1}_{0})$ are defined as
in Definition \ref{projectionPn},
with associated projections $P_{n}$. 
Suppose that the regularization hypothesis of Definition
\ref{reghyp} holds.
There is a pair $\delta_{0}, n_{0}$,
such that, in the closed ball ${\overline B(\Psi, \delta_{0})} \subset
C(J; H^{1}_{0})$, and for $n \geq n_{0}$,  
there
is a unique solution of the (nonlinear) Faedo-Galerkin equation 
(\ref{step2linFGal}), with $\rho = |u|^{2}$. 
If we designate the solutions by $u = \Psi_{n}$, then
these approximations converge to the unique solution $\Psi$ with order
$O(\|\Psi - P_{n}\Psi\|_{C(J; H^{1}_{0}}) +
O(\|\Psi_{0} - Q_{n}\Psi_{0}\|_{H^{1}_{0}})$.
This is the maximal expected order of convergence.
\end{theorem} 
\begin{proof}
We have verified the hypotheses of Theorem \ref{ther2} in sections 
\ref{hypI} and \ref{hypII}. This yields the unique local existence of the
nonlinear Faedo-Galerkin approximations $\Psi_{n}$ for sufficiently large
$n$.  
According to Remark \ref{remark2.1}, we have
$$
\|P_{n}K \Psi - K_{n} P_{n} \Psi\| \leq
\|(P_{n} - I) K\Psi\| + \| K\Psi - KP_{n}\Psi\| +
\|(K - K_{n})P_{n}\Psi\|. 
$$
The first two of these expressions are of the order
$O(\|I - P_{n}\|)$ in $C(J; H^{1}_{0})$.
In fact, since $\Psi$ is a fixed point of $K$, the 
convergence of the first term on the rhs
is a consequence of (\ref{new3}). 
The convergence of the second term follows from the local Lipschitz
property of $K$, combined with (\ref{new3}).  
The convergence of the third term follows from Proposition 
\ref{KnapproximatesK}, and is
of order 
$O(\|\Psi - P_{n}\Psi\|_{C(J; H^{1}_{0}}) +
O(\|\Psi_{0} - Q_{n}\Psi_{0}\|_{H^{1}_{0}})$.
This completes the proof.
\end{proof}
\section{Summary}
We have employed a powerful and precise operator calculus to obtain a
sharp convergence theory for Faedo-Galerkin approximations for solutions
of time dependent closed quantum systems with Kohn-Sham potentials.
Computation has not been discussed in the article. This topic was discussed
in \cite{JP}, where it was shown that a spectral method, embodied in the
algorithm FEAST, can accommodate the evolution operator approach,
extending to time and spatial discretization. The evolution operator is
especially appropriate for a fixed point 
approximation theory, as used here, because it
permits explicit analytical calculation. 

Natural extensions of the present study include the incorporation of time
discretization, and additional potentials to account for
exchange-correlation. The latter is particularly important to track charge
exactly, and thus the physical properties which use exact charge.

Since the evolution operator accommodates convergent iteration,
it would be especially valuable if the iteration analyzed in
\cite{archive} could be extended to standard finite dimensional
approximation theory.

We are not aware of any other systematic approach to Faedo-Galerkin
approximation for this model. For a special choice of $F_{n}$, based on
a smooth orthonormal system, the authors of \cite{SCB} demonstrate
convergence, as part of an existence analysis for a quantum control model.

Finally, we make the following observation, which paraphrases that of
\cite{KrasVain}. Although six separate convergence estimates are required for
the application of the theory, it is only those cited in Remark 
\ref{remark2.1} which govern the rate of convergence, established in
Theorem \ref{mainresult}. If $\Psi$ and $\Psi_{0}$ have additional
regularity, one expects these estimates to confirm increased order of
convergence from standard approximation theory, when applicable. 

\appendix
\section{Notation and Norms}
We employ complex Hilbert spaces 
in this article. 
$$
L^{2}(\Omega) = \{f = (f_{1}, \dots, f_{N})^{T}: |f_{j}|^{2} \;
\mbox{is integrable on} \; \Omega \}.
$$
$$
(f,g)_{L^{2}}=\sum_{j=1}^{N}\int_{\Omega}f_{j}(x){\overline {g_{j}(x)}} \; dx.
$$
However, $\int_{\Omega} fg$ is interpreted as 
$$
\sum_{j=1}^{N} \int_{\Omega} f_{j} g_{j} \;dx.
$$
For $f \in L^{2}$, as just defined, if each component $f_{j}$ satisfies
$
f_{j} \in H^{1}_{0}(\Omega; {\mathbb C}), 
$
we write $f \in H^{1}_{0}(\Omega; {\mathbb C}^{N})$, or simply, 
$f \in H^{1}_{0}(\Omega)$.
The inner product in $H^{1}_{0}$ is 
$$
(f,g)_{H^{1}_{0}}=
(f,g)_{L^{2}}+\sum_{j=1}^{N}\int_{\Omega} 
\nabla f_{j}(x) \cdotp {\overline {\nabla g_{j}(x)}} \; dx.
$$
$\int_{\Omega} \nabla f \cdotp \nabla g$ is interpreted as
$$
\sum_{j=1}^{N}\int_{\Omega} 
\nabla f_{j}(x) \cdotp \nabla g_{j}(x) \; dx.
$$
Finally, $H^{-1}$ is defined as the dual of $H^{1}_{0}$, and its
properties are discussed at length in \cite{Adams}. 
The Banach space $C(J; H^{1}_{0})$ is defined in the traditional manner:
$$
C(J; H^{1}_{0}) = \{u:J \mapsto H^{1}_{0}: u(\cdotp) \mbox{is
continuous}\}, 
$$
$$
\|u\|_{C(J; H^{1}_{0}} = \sup_{t \in J} \|u(t)\|_{H^{1}_{0}}.
$$

%\centerline{\large \bf References}

\end{document}